\documentclass[11pt,a4paper]{article}
\usepackage{cite}
\usepackage{amsmath}
\usepackage{amsthm}
\usepackage{amssymb}
\usepackage{amscd}
\usepackage{graphicx}
\usepackage{epsfig}
\usepackage{color}
\usepackage{mathrsfs}
\usepackage{mathtools}
\usepackage{amsfonts}
\usepackage{stmaryrd}
\usepackage{hyperref}

\newcommand{\X}{\mathfrak{X}}

\DeclarePairedDelimiter{\gen}{\langle}{\rangle}
\newcommand{\Cinfty}{\mathscr{C}^\infty}

\usepackage[left=2.5cm,right=2.5cm,top=3cm,bottom=3.5cm]{geometry}
\usepackage[colorinlistoftodos]{todonotes} 

\newtheorem{theorem}{Theorem}[section]
\newtheorem{proposition}[theorem]{Proposition}
\newtheorem{lemma}[theorem]{Lemma}
\newtheorem{corollary}[theorem]{Corollary}

\theoremstyle{definition}
\newtheorem{example}[theorem]{Example}
\newtheorem{counterexample}[theorem]{Counterexample}
\newtheorem{remark}[theorem]{Remark}
\newtheorem{definition}[theorem]{Definition}

\newcommand*{\NN}{\mathbb{N}}

\newcommand*{\ZZ}{\mathbb{Z}}

\newcommand*{\RR}{\mathbb{R}}
\newcommand*{\R}{\mathbb{R}}

\newcommand{\be}{\begin{equation}}
\newcommand{\ee}{\end{equation}}

\newcommand{\bea}{\begin{eqnarray}}
\newcommand{\eea}{\end{eqnarray}}
\newcommand{\beas}{\begin{eqnarray*}}
\newcommand{\eeas}{\end{eqnarray*}}
\newcommand{\w}{{\mathsf w}}

\newcommand{\mn}{{\medskip\noindent}}

\newcommand{\no}{{\noindent}}

\newcommand\restr[2]{{
  \left.\kern-\nulldelimiterspace 
  #1 
  \right|_{#2} 
}}
\newcommand{\T}{{\mathsf T}} 
\newcommand{\sT}{{\mathsf T}}
\newcommand{\cT}{\T^{\ast}} 

\newcommand*{\dd}{\mathrm{d}}

\DeclareMathOperator{\rank}{rank}
\DeclareMathOperator{\corank}{corank}
\DeclareMathOperator{\class}{class}
\newcommand*{\contr}[1]{\iota_{#1}}
\newcommand*{\liedv}[1]{\mathcal{L}_{#1}}
\newcommand*{\Reeb}{R}

\DeclareMathOperator{\Sec}{Sec}

\let\Im\Image

\newcommand{\n}{\nabla}
\newcommand{\pa}{\partial}
\def\sV{{\mathsf V}}

\def\End{\mathsf{End}}
\def\Lin{\mathsf{Lin}}

\def\la{\langle}
\def\ran{\rangle}

\def\xd{\mathrm{d}}
\newcommand{\dt}{{\xd_{\sT}}}
\newcommand{\dts}{\xd_{\sT^*}}
\newcommand{\za}{\alpha}
\newcommand{\zb}{\beta}
\newcommand{\zf}{\varphi}
\newcommand{\we}{\wedge}
\newcommand{\zi}{\iota}
\newcommand{\zw}{\omega}

\newcommand{\red}{\mathrm{red}}
\begin{document}
\title{\textbf{A Darboux classification of homogeneous\\ Pfaffian forms on graded manifolds}\footnote{The research of JG was funded by the National Science Centre (Poland) within the project WEAVE-UNISONO, No. 2023/05/Y/ST1/00043.}}
\date{}
\author{\\  \Large Janusz Grabowski
        \\ \Large Asier L\'opez-Gord\'on\\ \\
          {\it Institute of Mathematics}\\
                {\it Polish Academy of Sciences}
}
\maketitle
\begin{abstract}\noindent We study the local classification problem for differential Pfaffian forms on a supermanifold $M$ that are homogeneous with respect to a given homogeneity structure on $M$. The most familiar examples of homogeneity structures are those associated with vector bundle structures. Our aim is to show that, for a homogeneous form of fixed degree, there exist homogeneous Darboux coordinates. As a consequence, we obtain Darboux-type normal forms for homogeneous Pfaffian forms, recovering as special cases the classical Darboux theorem together with its contact and presymplectic counterparts.

To formulate an analogue of Darboux classification in the supergeometric setting, we associate to a differential form $\za$ the characteristic distribution
$\chi(\za)=\ker(\za)\cap\ker(\xd\za)$, and define the class of $\za$ as the rank of $\chi(\za)$. We prove that, under suitable regularity and constant-rank assumptions, this distribution completely determines the local equivalence problem for homogeneous Pfaffian forms. Our results apply equally well to ordinary (purely even) manifolds.

\par

\end{abstract}

\section{Introduction}
The celebrated local characterizations of contact and symplectic forms due to Darboux can, in fact, be viewed as special cases of a more general characterization of Pfaffian forms (i.e., nowhere-vanishing $1$-forms) of \emph{constant class}, in the terminology of Darboux.

\mn A Pfaffian form $\za$ on a manifold $M$ of dimension $n$ is said to be \emph{of constant class $(2s+1)$} (with $s=0,1,\ldots$) if
$$
\za\we\big(\xd\za\big)^s\neq 0\, , \qquad \big(\xd\za\big)^{s+1}=0\, .
$$
Darboux proved that such a form can always be written locally in the form
\be\label{e1}
\za=\xd x^0-\sum_{k=1}^s z^k \xd x^k,
\ee
for a suitable system of local coordinates
$$
\big(x^0,\dots,x^s,z^1,\dots,z^s,y^1,\dots,y^{n-2s-1}\big).
$$

\mn If $\za$ is \emph{of constant class $(2s+2)$}, i.e.,
$$
\za\we\big(\xd\za\big)^s(x)\neq 0,\qquad
\big(\xd\za\big)^{s+1}(x)\neq 0,\qquad
\za\we\big(\xd\za\big)^{s+1}(x)=0\, ,
$$
then $\za$ can be locally written as
$$
\za=z^0 \xd x^0-\sum_{k=1}^s z^k \xd x^k,
$$
for a suitable system of local coordinates
$$
\big(x^0,\dots,x^s,z^0,\dots,z^s,y^1,\dots,y^{n-2s-2}\big).
$$

\mn In particular, a nowhere-vanishing closed $1$-form has constant class $1$, so that \eqref{e1} reduces in this case to the Poincar\'e lemma. On the other hand, if $\dim M=2s+1$, then a contact form is precisely a Pfaffian form of constant class $(2s+1)$.

\mn Of course, Darboux's original characterization of the class does not extend directly to supermanifolds, where top-degree forms generally do not exist. For instance, if $\xi$ is an odd coordinate, then
$$
(\xd \xi)^s\neq 0
$$
for all $s=1,2,\dots$. A way around this difficulty is provided by an alternative definition of the class:
$$
\class(\za)\coloneqq \corank\big(\chi(\za)\big),
$$
where
$$
\chi(\za)=\ker(\za)\cap\ker(\xd\za)
$$
is the characteristic distribution of $\za$. This definition already appears in the work of Godbillon \cite{Godbillon:1969} (see also \cite{Grabowska:2024}) and applies naturally to supermanifolds. Moreover, it makes sense for arbitrary differential forms, not only Pfaffian forms. Conceptually, it clarifies that the class measures the degree of non-integrability of the distribution $\ker(\za)$, with the extreme cases corresponding to closed forms and contact forms, whose kernels are completely integrable and maximally non-integrable, respectively.

\mn In this paper, we study differential forms on supermanifolds that are additionally \emph{homogeneous}. Since our results are local, a homogeneity structure may be represented locally by a \emph{weight vector field}, i.e., a vector field of the form
\be\label{wvf}
\n=\sum_i w_i\,x^i\pa_{x^i},
\ee
in suitable local coordinates $(x^i)$, where $w_i\in\R$. A standard example is the Euler vector field
$$
\n=\sum_i y^i\pa_{y^i}
$$
associated with a vector bundle and written in bundle coordinates $(x^a,y^i)$. The weight vector field determines the notion of homogeneity for tensor fields: the coordinate $x^i$ is homogeneous of weight $w_i$, and a tensor field $K$ is said to be homogeneous of weight $w\in\R$ if
$$
\liedv{\n}K=wK.
$$
Here $\liedv{}$ denotes the Lie derivative. Note that the weights $w_i$ are allowed to be arbitrary real numbers, which goes far beyond the usual $\NN$- or $\ZZ$-gradings.

\mn All these notions extend naturally to supermanifolds with local coordinates $(x^i)$, some even and some odd. The weight vector field is always even. In addition to the weight, one also has parity, so homogeneous tensors of \emph{degree}
$$
\lambda=(\sigma,w)\in\ZZ_2\times\R
$$
are understood to be tensors of parity $\sigma$ and weight $w$.

\mn Homogeneous differential forms in this sense arise in many geometric situations. A particularly important class is provided by homogeneous symplectic structures on so-called \emph{N-manifolds}, introduced independently by Roytenberg \cite{Roytenberg:2002} and {\v{S}}evera \cite{Severa:2005} (see also \cite{Voronov:2002}). These are homogeneity supermanifolds whose local coordinates have non-negative integer weights compatible with parity. Such structures lead to many interesting geometric objects and applications. Already for weights restricted to the values $0,1,2$, Roytenberg obtained in \cite{Roytenberg:2002} an elegant description of \emph{Courant algebroids} \cite{Liu:1997} as Hamiltonian systems generated by an odd homological Hamiltonian satisfying
$$
\{H,H\}=0.
$$
The aim of this paper is to develop a local theory of homogeneous differential forms and, in particular, to establish Darboux-type theorems in terms of homogeneous coordinates. Homogeneous differential forms on vector bundles, studied by Tulczyjew and Urba\'nski in \cite{Tulczyjew:2000}, are polynomial. The existence of homogeneous Darboux coordinates is therefore a natural question, but one that is far from trivial. Indeed, if a homogeneous Pfaffian form has constant class $(2s+1)$, then the classical Darboux theorem guarantees a local representation of the form \eqref{e1}. However, the coordinates appearing there need not themselves be homogeneous. Can one always choose homogeneous Darboux coordinates?

\mn As an illustration, consider the cylinder
$$
M=\big\{(z,p,q)\in\R^3\,:\, p^2+q^2<1\big\},
$$
where the coordinates $(z,p,q)$ are assigned weights $(0,1,-1)$. The $1$-form
$$
\za=\xd z-p\big(2+\sin(pq)\big)\xd q
$$
is then a homogeneous contact form of weight $0$, although it is not polynomial. What are homogeneous Darboux coordinates for this form? The results of this paper provide answers to such questions for broad classes of homogeneous differential forms on general homogeneity supermanifolds.

\mn The paper is organized as follows. In Section~\ref{sec:homogeneity}, we review the basic notions related to homogeneity supermanifolds. Section~\ref{sec:VBs} contains a characterization of vector superbundles in terms of homogeneity structures, together with a Frobenius theorem for homogeneous distributions. In Section~\ref{sec:Darboux}, we prove Darboux theorems for homogeneous Pfaffian forms. Finally, Section~\ref{sec:examples} is devoted to examples and applications.

\section{Introduction to homogeneity supermanifolds}\label{sec:homogeneity}
In this section, we recall the main notions and results about homogeneity supermanifolds from \cite{Grabowska:2025} which will be employed subsequently. The fundamental concepts are: the degree of homogeneity of a tensor field and the homogeneity of a distribution associated with a \emph{weight vector field}. In addition, we present homogeneous versions of Poincar\'e Lemma and Frobenius Theorem. The reader not familiar with the language of supergeometry can skip the parts related to the parity; all works for standard (purely even) manifolds.

\mn Given an even vector field $\nabla$ on a supermanifold $M$, a (local) smooth superfunction $f$ on $M$ is called a (local) \emph{homogeneous function of weight $w\in\RR$} (with respect to $\n$) if $\nabla(f)=w\cdot f$. If, in addition, $f$ has the parity $\sigma\in \ZZ_2$, such a function is called \emph{homogeneous of degree $\lambda=(\sigma,w)\in\ZZ_2\times\RR$}, and we write $\w(f)=w$, $\deg(f)=\lambda$. The function $f$ (or the degree $\lambda$) is called \emph{even} if $\sigma=0$ and \emph{odd} if $\sigma=1$. Similarly, a tensor field $K$ is of weight $w\in\R$ if $\liedv{\n} K=w\cdot K$, where $\liedv{}$ denotes the Lie derivative. From properties of the Lie derivative it easily follows that the tensor product $K_1\otimes K_2$ of homogeneous tensor fields of weights $w_1$ and $w_2$ is homogeneous of weight $w_1+w_2$.

\begin{definition}
An even vector field $\nabla$ on a supermanifold $M$ with body (reduced manifold) $|M|$ is called a \emph{weight vector field} if, in a neighbourhood of every $p\in|M|$, there exists a system of local coordinates which are homogeneous with respect to $\n$. In other words, there is an atlas of charts $(U_\alpha, (x^i_\alpha))$ on $M$ such that
  \be\label{eq:weight_vector_field} \nabla=\sum_iw_i^\alpha \cdot x^i_\alpha\,\pa_{x^i_\alpha}\,\ee
  where $w_i^\alpha\in\RR$ represents the weight of the homogeneous coordinate $x^i_\alpha$.
  These coordinates are not assumed to vanish at $p$.

  A chart with homogeneous coordinates is called a \emph{homogeneity chart}. A supermanifold endowed with a weight vector field is called a \emph{homogeneity supermanifold}.
\end{definition}

\begin{example}
  Let $x$ be the canonical coordinate on $\RR$. The vector field $\nabla_1=x\,\pa_x$ is a weight vector field, whereas $\nabla_2=x^2\,\pa_x$ is not a weight vector field. Indeed, if $y=y(x)$ is a homogeneous coordinate of weight $w\in\RR$ in a neighbourhood of $0\in\RR$, then $y'(0)=a\neq 0$ and $\nabla_2(y)=x^2\cdot y'(x)=w\,y(x)$. The derivative of the left side is 0 at $0\in\RR$, which implies that $w=0$, but $x^2\,y'(x)=0$ has only constant solutions $y(x)$. Nevertheless, in a neighbourhood of $0\in\RR$ there exist non-constant homogeneous functions with respect to $\nabla_2$. For instance, the function
  $$f(x)=\begin{cases} 0 & \text{if} \quad x\le 0\\ e^{-1/x} & \text{if} \quad x>0\end{cases}$$
  is homogeneous with respect to $\nabla_2$ with weight 1.
\end{example}
\begin{example}\label{example:standard_superspace_Euler}
Consider the supermanifold $\R^{p|q}$ with the standard linear coordinates $(y^i)$ (some of them even and some odd). The \emph{Euler vector field}
$$\n^{p|q}=\sum_iy^i\pa_{y^i},$$
is a particular weight vector field, for which the coordinates $(y^i)$ are homogeneous of weight 1.

A useful observation is that the even vector field $\n^{p|q}$ is exactly the generator of the multiplication by positive reals in $\R^{p|q}$: if $h_t(y^i)=(e^ty^i)$
for $t\in\R$, then $h_t$ is a one-parameter group of diffeomorphisms of $\R^{p|q}$ and
$$\n^{p|q}=\restr{\frac{\xd}{\xd t}}{t=0}h_t.$$
Therefore, the Euler vector field $\n^{p|q}$ is completely determined by the multiplication by positive reals in $M=\R^{p|q}$. Actually, there is a unique smooth extension of the one-parameter group of diffeomorphisms $h_t$ into the action $\mathsf{h}_s(x^i)=(sx^i)$ of the monoid $(\R,\cdot)$ of multiplicative reals.

Like in the classical {Euler's Homogeneous Theorem}, we can prove that $f$ is a homogeneous function of weight 1 on $\R^{p|q}$ if and only if $f=\sum_i F_i{y^i}$ for some $F_i\in\R$. In other words, $f$ is a linear function in the standard sense. In other words, \emph{linear functions} on the supermanifold $\R^{p|q}$ are exactly smooth functions which are 1-homogeneous with respect to $\n^{p|q}$.
\end{example}

\no Let us recall that if an even vector field $X$ on a supermanifold $M$ of dimension $(n|m)$ vanishes at $p\in|M|$, then there is a well-defined differential $D_pX\in\End(\T_pM)\simeq\operatorname{gl}(n,m)$. Indeed, we can view $D_pX$ as $\T_pX:\T_p M\to\sV_{0_p}\T M$, identifying canonically the vertical part $\sV_{0_p}\T M$ of the tangent space $\T_{0_p}\T M$ with $\T_pM$. It is worth remarking that this works only because $X(p)=0$. In this case, the condition \eqref{eq:weight_vector_field} implies that $D_p\n$ is diagonalizable.

On the other hand, if $X(p)\neq 0$, then there exists a system of local coordinates $(x^a)$ around $x_0$ straightening $X$, namely, $X=\pa_{x^1}$. Hence, it is possible to construct a system of homogeneous coordinates $(y^a)$ with arbitrary degrees $\{w_a\}\subset \RR$, with the only restriction that $w_1$ is non-zero:
$$y^1 = e^{w_1\, x^1}\, , \quad y^i = e^{w_i\, x^1} x^i \, ,\quad i\geq 2\, .$$
Note that the system of straightening coordinates $(x^a)$ always exist for an even vector field around a point where it is non-zero. For an odd vector field $X$, it is additionally required to assume the integrability condition $[X,X]=0$ \cite{Shander:1980} (see also \cite[Chapter 4.4]{Manin:1997}). We refer to \cite[Chapter 6]{Carmeli:2011} and references therein for more details.

\begin{proposition}\label{proposition:homogeneous_coordinates}
An even vector field $\n$ on a supermanifold $M$ is a weight vector field if and only if $\n$ is locally linear around $m$ and $D_m\n$ is diagonalizable, for each $m\in|M|$ in the zero-locus of $\n$. In particular, non-vanishing even vector fields are weight vector fields.

\mn If, in local homogeneous coordinates $(x^i)$, we have
$$\nabla=\sum_{i=1}^nw_i\cdot x^i\,\pa_{x^i}\, ,$$
where $n$ is the total dimension of $M$, then:
\begin{itemize}
\item In the case $\nabla(m)=0$, all weights of systems of homogeneous coordinates around $m$ are
the same for each system of homogeneous coordinates, up to permutations among the weights of even and odd coordinates separately.
In particular, if $\,\nabla$ admits local homogeneous coordinates with weights in a subset $\Gamma\subset\RR$ (e.g., with integer weights), then all homogeneous coordinate systems around $m$ have weights in $\Gamma$ (have integer weights).

\item In the case $\nabla(m)\neq 0$, a vector $(w'_i)\in\RR^{n}$ consists of weights of a system of local homogeneous coordinates in a neighbourhood of $m$ if and only if not all weights of even coordinates are 0. In particular, for any $w,v_i\in\RR$, $i=2,\dots,n$, $w\neq 0$, there are coordinates $(x,z^i)$, $i=2,\dots, n$ in a neighbourhood of $m$, such that $x$ is even, $x(m)=1$, $z^i(m)=0$, and
    $$\n=w\cdot x\,\pa_x+\sum_{i=2}^nv_i\cdot z^i\,\pa_{z^i}\,.$$
A particular choice we will use in the sequel is $\n = x \pa_x$, with $x(m) =1$.
\end{itemize}
\end{proposition}

\begin{remark}\label{remark:zero_homogeneous_functions}
  If $f$ is a function on a homogeneity supermanifold $(M, \nabla)$ with non-zero weight $\w(f)=w\in \RR$, then it vanishes at every point $x_0\in |M|$ at which $\nabla$ vanishes. Indeed,
  $$f(x_0) = \frac{1}{w} \left(\liedv{\nabla} f\right)  (x_0) =  \frac{1}{w} \left(i_{\nabla} \dd f\right) (x_0) = 0\,.$$
  In other words, if a smooth function $f$ is non-vanishing at $x_0$ and has weight $w=\w(f)$, then $w = 0$. This observation will be useful in several subsequent proofs.

  One of the consequences of this fact is that if a homogeneous one-form $\omega$ does not vanish at $x_0$, then $\w(\omega) \in \Gamma$, with $\Gamma\subset \RR$ the set of weights of any system of homogeneous coordinates $(x^a)$ around $x_0$ (see Proposition~\ref{proposition:homogeneous_coordinates}). Otherwise, all the coefficients of $\omega$ in the basis $(\dd x^a)$ would vanish at $x_0$. More generally, if $\beta$ is a homogeneous $k$-form non-vanishing at $x_0$, then
  $$\w(\beta) = \sum_{i=1}^k \lambda_i \, , \quad \lambda_i \in \Gamma\, .$$
  Similarly, if $X$ is a homogeneous vector field such that $X(x_0)\neq 0$, then $\w(X)\in - \Gamma$.
\end{remark}

\begin{definition}
  A \emph{homogeneous submanifold} of a homogeneity supermanifold $(M,\nabla)$ is a submanifold $N\subset M$ to which the weight vector field $\n$ is tangent. Consequently, $(N,\n\,\big|_N)$ is a homogeneity manifold itself.
\end{definition}

 \begin{definition}[Homogeneity category]
  A \emph{morphism} of a homogeneity supermanifold $(M,\nabla_M)$ into a  homogeneity supermanifold $(N,\nabla_N)$ is a morphism $\varphi:M\to N$ of supermanifolds that relates $\nabla_M$ with $\nabla_N$. In other words,
  the pullbacks of (local) homogeneous functions on $N$ are homogeneous on $M$, with the same degree. Homogeneity supermanifolds with these morphisms form the category $\mathsf {HSMan}$.
\end{definition}

\section{Vector superbundles}\label{sec:VBs}
Vector superbundles are fundamental examples of homogeneity supermanifolds. Moreover, the homogeneity approach to vector bundles is simpler and much more effective than the standard ones. It was originated in \cite{Grabowski:2009} and generalized to \emph{graded bundles} in \cite{Grabowski:2012}. A supergeometric version of the latter can be found in \cite{Jozwikowski:2016}.

\mn A \emph{vector superbundle of rank $(k,l)$} is a fiber superbundle $\tau:V\to M$ with the typical fiber $\R^{k|l}$, and such that the transition maps
$$\phi_{\za\beta}:U_\za\cap U_\beta\times\R^{k|l}\to U_\za\cap U_\beta\times\R^{k|l},$$
associated with local trivializations $\tau^{-1}(U_\za)\simeq U_\za\times\R^{k|l}$ of the fiber bundle structure, respect the \emph{Euler vector field}
$$\n^{k|l}_{U_{\za\zb}}=\sum_iy^i\pa_{y^i},$$
on $U_{\za\zb}\times\R^{k|l}$, where $(y^i)$ are the standard linear coordinates on $\R^{k|l}$ viewed as a part of coordinates $(x^a,y^i)$ on $U_{\za\zb}\times\R^{k|l}$. Here, $U_{\za\zb}=U_\za\cap U_\beta$. In other words, $\n^{k|l}_{U_{\za\zb}}$ is tangent to the fibers, where it coincides with $\n^{k|l}$. This gives rise to a globally defined Euler vector field $\n_V$ on $V$ which is a particular weight vector field, so $(V,\n_V)$ is a homogeneity supermanifold. A \emph{vector subbundle} in $V$ is a homogeneous submanifold of the latter, generally supported on a submanifold of $M$.

\mn \emph{Linear functions} on $V$ are just smooth 1-homogeneous (super)functions. \emph{Basic functions}, in turn, are homogeneous functions with weight 0 -- the pullbacks of (super)functions on $M$. The space $\Lin(V)$ of linear functions is canonically a bi-module over $\Cinfty(M)$, and every linear function $f$ has two linear local forms,
$$f=F_i(x)\cdot y^i=y^i\cdot G_i(x).$$
On supermanifolds, generally $F_i\ne G_i$, as (super)functions do not commute. In other words, the left- and the right-module structures are generally different.

\mn As the transition functions map linear functions into linear ones, they take the form
$$\phi_{\za\beta}\big(x^a,y^i\big)=\Big(x^a,\big[\phi_{\za\beta}\big]^i_j(x)\cdot y^j\Big).$$
The \emph{dual vector superbundle}, $\pi:V^*\to M$, is obtained from the same local trivializations with local linear coordinates $(x^a,z_i)$, but with the \emph{dual transition maps} $\phi^*_{\za\zb}$,
$$\phi^*_{\za\beta}\big(x^a,z_i\big)=\Big(x^a, z_j\cdot \big[\phi_{\zb\za}\big]^j_i(x)\Big).$$
This implies that the local pairing
$$\big\langle A^i(x)\cdot z_i;y^j\cdot B_j(x)\big\rangle=A^i(x)B_i(x)$$
between linear functions on $V^*$ and $V$ is a well defined global map
\be\label{lp}\big\langle\cdot ;\cdot\big\rangle: \Lin(V^*)\times\Lin(V)\to \Cinfty(M),\ee
that has a tensorial character,
$$\langle F(x)f;g\rangle=F(x)\langle f;g\rangle,\quad \langle f;gF(x)\rangle=\langle f;g\rangle F(x).$$
Of course, $V^{**}=V$.

\mn In the standard terminology, \emph{sections} of the vector bundle $\tau:V\to M$ are understood as linear functions on the dual bundle $V^*$, $\Sec(V)=\Lin(V^*)$, and this identification is usually denoted $X\mapsto \zi_X$. Note that sections in this sense cannot be generally identified with morphisms $X:M\to V$ such that $\tau\circ X=\operatorname{id}_M$.
The local free generators $(z_i)$ of the $\Cinfty(M)$-module $\Lin(V^*)$ are viewed as a \emph{basis} $(e_i)$ of local sections of $V$. The \emph{dual basis} consists of $(e^i_*=y^i)$. We have
$$\langle e_i,e^j_*\rangle=\delta^j_i,$$
so the pairing (\ref{lp}) is understood as the contraction
$$i_X\zw=\zi_X(\zw)=\big\langle \zi_X,\zi_\zw\big\rangle.$$

\begin{example} A canonical example of a vector superbundle is the \emph{tangent bundle} $\sT M$ of a supermanifold $M$ of dimension $(k,l)$. It is defined as a vector superbundle $\tau_M:\sT M\to M$ with the typical fiber $\R^{k|l}$. With every coordinate chart $\zf_\za:U_\za\to\R^{k|l}$ on $M$ with corresponding (super)coordinates $(x^i_\za)$, we associate a local trivialization
$$\phi_\za:\tau_M^{-1}(U_\za)\to U_\za\times\R^{k|l}$$
and the associated coordinates $(x^i_\za,\dot x^j_\za)$, declaring the transition functions
$\phi_{\za\zb}=\phi_\za\circ\phi_\zb^{-1}$ to be
\begin{align*}
& x_\za=\zf_{\za\zb}(x_\zb); \\
& \dot{x}_\za^i = \sum_j \frac{\partial x^i_\za}{\partial x^j_\zb}(x) \, \dot{x}^j_\zb\, ,
\end{align*}
where $\zf_{\za\zb}=\zf_\za\circ\zf_\zb^{-1}$ are the transition maps for the supermanifold $M$. It is easy to see that these transition maps are linear in $(\dot x^i)$, so they respect the Euler vector field
$$\n=\sum _i\dot x^i\cdot\pa_{\dot x^i}.$$
The system of coordinates $(x^i,\dot x^j)$ on $\sT M$ is called \emph{adapted} to the system of coordinates $(x^i)$ on $M$.

\mn The dual bundle is the \emph{cotangent bundle} $\sT^*M$, with the adapted coordinates $(x^i,p_j)$. In the standard terminology, $p_j$ are denoted $\pa_{x^j}$, and form a local basis of the $\Cinfty(M)$-module of vector fields on $M$. This concerns both module structures on the space of vector fields. Similarly, $\dot x^i$ are denoted $\xd x^i$, providing a local basis of the bi-module of differential forms on $M$.

\mn By a \emph{distribution} $D$ of rank $k$ on $M$ we will understand a vector subbundle in $\T M$ of rank $k$, i.e., freely generated by $k$ nowhere-vanishing vector fields.
A distribution $D$ is called \emph{even} (respectively, \emph{odd}) if it is locally freely generated by even (respectively, odd) vector fields. This characterization does not depend on the choice of local free generators \cite{Grabowska:2025}. Even and odd codistributions are defined analogously.
\end{example}

\begin{theorem}[Lifts of homogeneity structures]
Let $M$ be a homogeneity supermanifold with the weight vector field $\nabla$, which in homogeneous coordinates reads
$$\nabla=\sum_aw_a\cdot x^a\pa_{x^a}\,.$$
Then, $\sT M$ and $\sT^*M$ are canonically homogeneity supermanifolds with the weight vector fields written in the adapted coordinates as
\be\label{tl}\dt\nabla=\sum_aw_a\left(x^a\pa_{x^a}+\dot x^a\pa_{\dot x^a}\right)\ee
and
\be\label{cl}\dts\nabla=\sum_aw_a\left(x^a\pa_{x^a}-p_a\pa_{p_a}\right)\,.\ee
\end{theorem}
\no Let us observe that the lifted weight vector fields $\dt\nabla$ and $\dts\nabla$ are particular cases of the complete lifts of vector fields on $M$ to the tangent $\sT M$ and cotangent $\sT^*M$ bundle, as defined in the literature in the purely even case (see \cite{Grabowski:1995,Yano:1966}). On supermanifolds, the lifted vector fields are defined, e.g., in \cite[Definition 5.2]{Grabowska:2025}.
Moreover, since $\dt \nabla$ is locally linear, it also defines a weight vector field on the shifted tangent bundle $\T[1]M$ of $M$.

\begin{example}\label{example:nabla_differential_forms}
  Consider a manifold $M$ with local coordinates $(x^i)$ and the weight vector field $\n =\sum_i x^i\pa_{x^i}$. Its tangent lift is given by
  $$\dt \n =\sum_i\left(x^i\pa_{x^i} + \dot x^i\pa_{\dot x^i} \right)\, .$$
  By changing the parity of the fiber coordinates, we can define an associated vector field on  the shifted tangent bundle $\T[1]M$:
  $$\overline \n = \sum_i\left(x^i\pa_{x^i} + \xi^i\pa_{\xi^i} \right)\, ,$$
  which is also a weight vector field. The algebra of smooth superfunctions on $\T[1]M$ is isomorphic to the algebra $\Omega^\bullet(M)$ of differential forms on $M$. This fact allows identifying $\xi^i \equiv \dd x^i$ for each $i$, and consequently
  $$\overline \n = \sum_i\left(x^i\pa_{x^i} + \xi^i\pa_{\xi^i} \right)
  \equiv \sum_i \left(x^i \liedv{\pa_{x^i}} + \dd x^i\wedge \contr{\pa_{x^i}}\right)
  = \liedv{\sum_i x^i \pa_{x^i}} = \liedv{\n}\, .
  $$
  Thus, $\n$ induces the weight vector field $\liedv{\n}$ on the supermanifold $(M, \mathcal{A}=\Omega^\bullet(M))$.
\end{example}

\begin{definition}
A distribution $D\subseteq \T M$ (respectively, a codistribution $D\subseteq \cT M$) on a homogeneity supermanifold $(M, \nabla)$ is called \emph{homogeneous} if $\dd_{\T} \nabla$ (respectively, $\dd_{\cT} \nabla$) is tangent to $D$, i.e., $D$ is a homogeneous submanifold in $\sT M$ (resp., $\sT^*M$).
\end{definition}

\mn Clearly, a distribution is homogeneous if and only if its annihilator is a homogeneous codistribution.
The following two theorems from \cite{Grabowska:2025} characterize homogeneous distributions.

\begin{theorem}\label{homdist}
Let $D\subset\sT M$ be a distribution on a homogeneity manifold $(M, \nabla)$. Then, $D$ is  homogeneous if and only if $D$ is locally generated by homogeneous vector fields.
\end{theorem}


\begin{theorem}[Homogeneous Frobenius Theorem]\label{theorem:Frobenius}
Let $D\subset\sT M$ be a rank-$k$ involutive homogeneous distribution on a homogeneity supermanifold $(M, \nabla)$ of total dimension $n$, and let $m\in|M|$. Then, there is a neighbourhood of $m$ endowed with a system $(x^i)$, $i=1,\dots,n$, of homogeneous local coordinates such that
\begin{description}
\item{(a)} $D$ is locally spanned by $\la\pa_{x^1},\dots,\pa_{x^k}\ran$, if $\n(m)=0$;
\item{(b)} $\n=x^n\pa_{x^n}$, and $D$ is locally spanned by  $\la\pa_{x^1},\dots,\pa_{x^k}\ran$, if $\n(m)\ne 0$ and $\n(m)\notin D_m$.
\item{(c)} $\n=x^n\pa_{x^n}$, and $D$ is locally spanned by $\la\pa_{x^1},\dots,\pa_{x^{k-1}},Y\ran$, where
$$Y=\n+\sum_{j=k}^{n-1}h_j(x^k,\dots,x^{n-1})\pa_{x^j},$$
in the case $\n(m)\ne 0$ and $\n(m)\in D_m$.

\end{description}
\end{theorem}
\no Note that, in the cases (a) and (b), the projection into the space of leaves $M/D$ locally is simply `forgetting' the first $k$ coordinates.

\mn In standard differential geometry, it is well known that linearly independent and pairwise commuting vector fields can be written as coordinate vector fields. We will prove a homogeneous version of this theorem.

\begin{lemma}\label{lemma:homog_sol_diff_eq}
  Let $(M, \n)$ be a $(p\vert q)$-dimensional homogeneity supermanifold with homogeneous coordinates $(y^a, \xi^\alpha),\, 1\leq a \leq p\, , 1\leq \alpha \leq q$. Consider the differential equation
  $$\frac{\partial f}{\partial y^1} (y, \xi) = g(y, \xi)\, ,$$
  where $g$ is a homogeneous function of degree $\lambda = (\sigma, w)\in \ZZ_2\times \RR$. Then, in a neighbourhood of every point $x_0 \in \vert M \vert$ at which $\n$ vanishes, this equation has a smooth solution given by
  $$f(y, \xi) = \int_0^{y^1} g\left(s, y^2, \ldots, y^p, \xi^1, \ldots, \xi^q\right)\, \dd s\, ,$$
  which satisfies $\w(f) = w + \w (y^1)$ and $f(x_0) = 0$.
\end{lemma}

\begin{proposition}[Homogeneous straightening coordinates]\label{proposition:straightening}
  Let $(M, \n)$ be a homogeneity supermanifold of total dimension $n$. Let $X_1, \ldots, X_k$ (with $1\leq k\leq n$) be linearly independent homogeneous vector fields which pairwise supercommute, i.e.,
  $$[X_i, X_j] = 0\, , \quad 1\leq i,j\leq k\, .$$
  Assume that, at a point $x_0\in |M|$, either $\n(x_0)=0$ or $\n(x_0)$ is linearly independent of $X_1(x_0), \ldots, X_k(x_0)$.
  Then, there exists a system of homogeneous coordinates $(x^a)$ around $x_0$ such that
  $$X_i = \pa_{x^i}\, ,\quad 1\leq i\leq k\, .$$
\end{proposition}

\begin{proof}
We shall prove it by induction on $k$. For $k=1$, we can consider the homogeneous involutive distribution $D = \gen{X_1}$. By Theorem~\ref{theorem:Frobenius}, around any $x_0\in |M|$, there exists a system of homogeneous coordinates $(\bar{x}^a)$ such that $D=\gen{\pa_{\bar{x}^1}}$. Therefore, $X_1 = f \pa_{\bar{x}^1}$ for some nowhere-vanishing even function $f$. This function is necessarily homogeneous, because both $\pa_{\bar{x}^1}$ and $X_1$ are homogeneous vector fields. Additionally, in the case $\n(x_0)=0$, $f$ has weight $0$ (see Remark~\ref{remark:zero_homogeneous_functions}). Note that $1/f$ is a homogeneous function with weight $\w(1/f) = -\w(f)$. Consider the homogeneous local function $x^1$ given by
$$x^1 = \int_{\bar{x}^1(x_0)}^{\bar{x}^1} \frac{\dd s}{f(s, \bar{x}^2, \ldots, \bar{x}^n)}\, ,$$
so that $\liedv{X_1} x^1 = 1$. It is functionally independent of $\bar{x}^2, \ldots, \bar{x}^n$. Thus, $(x^1, x^2 = \bar{x}^2, \ldots, x^n = \bar{x}^n)$ is a system of homogeneous coordinates in which $X = \pa_{x^1}$.

\mn Now consider $k$ linearly independent homogeneous vector fields $X_1, \ldots, X_k$ which pairwise supercommute. Applying the induction assumption, there exists a system of homogeneous coordinates $(\tilde{x}^a)$ such that
$$X_i = \pa_{\tilde{x}^i}\, , \quad 1\leq i \leq k-1\, .$$
Taking into account that $X_k$ is homogeneous and supercommutes with $X_1,\ldots, X_{k-1}$, we can see that, in these coordinates, it reads
$$X_k = \sum_{i=1}^{k-1} A_i(\tilde{x}^k, \ldots, \tilde{x}^n) \pa_{\tilde{x}^i} + \sum_{\mu=1}^{n} B_\mu(\tilde{x}^k, \ldots, \tilde{x}^n) \pa_{\tilde{x}^\mu}\, ,$$
where $A_i$ and $B_\mu$ are homogeneous functions that depend only on the coordinates $(\tilde{x}^k, \ldots, \tilde{x}^n)$. Moreover, at least one of the functions $B_\mu$ is nowhere-vanishing, because $X_k$ is linearly independent of the other vector fields. Therefore,
$$Y \coloneqq \sum_{\mu=1}^{n} B_\mu(\tilde{x}^k, \ldots, \tilde{x}^n) \pa_{\tilde{x}^\mu}$$
is a nowhere-vanishing homogeneous local vector field, and it depends only on the coordinates $(\tilde{x}^k, \ldots, \tilde{x}^n)$. As in the case $k=1$, we can make a change of homogeneous coordinates
$(\tilde{x}^k, \ldots, \tilde{x}^n)\mapsto (x^k, \ldots, x^n)$
so that $Y= \pa_{x^k}$. Then,
$$X_k = \sum_{i=1}^{k-1} A_i(\tilde{x}^k, \ldots, \tilde{x}^n) \pa_{\tilde{x}^i} +  \pa_{x^k}\, ,$$
in the system of homogeneous coordinates $(\bar{x}^1, \ldots, \bar{x}^{k-1}, x^k, \ldots, x^n)$. Finally, we can define the homogeneous coordinates
$$x^i = \bar{x}^i - \int_{x^k(x_0)}^{x^k} A_i \left(s,\tilde{x}^{k+1}, \ldots, \tilde{x}^n\right)\, \dd s\, , \quad 1\leq i \leq k-1\, ,$$
which satisfy
$$\liedv{X_j} x^i = \delta_j^{\, i}\, , \quad \liedv{X_k} x^i = 0\, , \quad 1\leq i, j\leq k-1\, .$$

\end{proof}

\section{Darboux theorems for homogeneous forms}\label{sec:Darboux}
On homogeneity supermanifolds, the homogeneous differential forms are of particular interest. We have the following homogeneous version of Poincar\'e Lemma (see \cite{Grabowska:2025}), that immediately leads to the corresponding homogeneous de Rham cohomology (for a linear version see \cite{Cabrera:2017}).
\begin{lemma}[Homogeneous Poincar\'e Lemma]\label{lemma:Poincare}
  Let $(M,\nabla)$ be a homogeneity supermanifold of total dimension $k$, let $m\in|M|$, and let
  $\nabla=\sum_{i=1}^kw_i\cdot x^i\,\pa_{x^i}$ in a neighbourhood of $m$. In this neighbourhood, take a closed homogeneous $n$-form $\omega$, $n>0$, of degree $\lambda=(\sigma,w) \in \ZZ_2 \times \RR$.
  \begin{enumerate}
    \item
    If $\nabla(m)=0$, then we can find a homogeneous $(n-1)$-form $\alpha$ of degree $\lambda$ such that $\omega=\dd\alpha$. Moreover, we can choose $\alpha$ such that $\alpha(m)=0$.
    \item
    If $\nabla(m)\neq 0$, then there is such a form $\alpha$ for $n>1$. For $n=1$, we can find a homogeneous function $f$ such that $\omega=\dd f$, except for the case $w=0$. However, the function $f$ does not vanish at $m$ (except $\omega=0$, of course).
    \item If $\nabla(m)\neq 0$ and a homogeneous 1-form $\omega$ has weight 0, then there are local homogeneous coordinates $(x,y^i)$ such that $\n = x \pa_x,\, x(m)=1$ and $y^i(m)=0$ (see Proposition~\ref{proposition:homogeneous_coordinates}). In such coordinates,
    $$\omega(x,y)=c\, \frac{\dd x}{x}+\dd g(y),$$
    for a constant $c\in\RR$, and $g=g(y)$ being a function depending on coordinates $(y^i)$ only. In other words,
    $$\omega=\dd f,\quad f(x,y)=c\log(x)+g(y)+c_1 \, ,$$
    for a constant $c_1\in \RR$. The function $f$ is not homogeneous if $c\neq 0$.
  \end{enumerate}
\end{lemma}

\mn Let now $M$ be a purely even manifold, and let $\za$ be a nowhere-vanishing 1-form. Darboux defined in \cite{Darboux:1882} certain classes of such forms, called \emph{Pfaffian forms}, as they define \emph{Pfaffian equations} $\za(X)=0$ (see also \cite[Chapter V]{Libermann:1987}).
\begin{definition}
A Pfaffian form $\za$ is said to be at $x\in M$
\begin{enumerate}
\item \emph{of class $(2s+1)$} if
\be\label{oddc} \za\we\big(\xd\za\big)^{s}(x)\ne 0,\quad \big(\xd\za\big)^{s+1}(x)=0;
\ee
\item \emph{of class $(2s+2)$} if
\be\label{evenc} \za\we\big(\xd\za\big)^{s}(x)\ne 0,\quad \big(\xd\za\big)^{s+1}(x)\ne 0,\quad \za\we\big(\xd\za\big)^{s+1}(x)= 0.
\ee
\end{enumerate}
\end{definition}
\no This characterization served Darboux to obtain a local description of Pfaffian forms of constant class.
\begin{theorem}[Darboux] Let $\za$ be a Pfaffian form on a manifold of dimension $n$. Then, around any point $x\in M$ there are local coordinates $(x^0,\dots,x^n)$ such that
\begin{enumerate}
\item
\be\label{oddc1}\za=\xd x^0-\sum_{k=1}^sx^{s+k}\xd x^k\ee
if $\za$ is of constant class $(2s+1)$;
\item
\be\label{evenc1}\za=x^{2s+1}\xd x^0-\sum_{k=1}^sx^{s+k}\xd x^k\ee
if $\za$ is of constant class $(2s+2)$.
\end{enumerate}
\end{theorem}
\no Of course, the Darboux's class definitions do not make much sense on supermanifolds, where the top forms generally do not exist, $\big(\xd\xi\big)^s\ne 0$ for any $s=1,2,\dots$, if $\xi$ is an odd variable.

\mn Fortunately, there is another definition of class, observed already by Godbillon \cite{Godbillon:1969} (see also \cite{Grabowska:2024}), that can be applied also to supermanifolds. We will use the standard Cartan calculus for differential forms on a supermanifold associated with the \emph{Deligne's sign convention} (see, e.g., \cite[Section 8]{Grabowska:2025} and \cite[Appendix to \S 1]{Deligne:1999}).

\begin{definition}
  Let $\alpha$ be a $k$-form on a supermanifold $M$ of total dimension $n$. The subset
  $$\chi(\alpha) = \ker(\alpha) \cap \ker (\dd \alpha) \subseteq \T M$$
  is called the \emph{characteristic set} of $\alpha$. Here, $\ker(\zb)$ is the left $\Cinfty$-module of vector fields $X$ such that $i_X\zb=0$.

 \mn  If $\chi(\alpha)$ is a distribution of rank $k$, i.e., it is locally freely generated by $k$ nonvanishing vector fields, it is called the \emph{characteristic distribution} of $\alpha$. In that case, we say that $\alpha$ is \emph{regular}, and the corank $(n-k)$ of $\chi(\alpha)$ as a sub-bundle of $\T M$ is called the \emph{class of $\alpha$}:
  $$\class(\alpha) \coloneqq \corank\big(\chi(\alpha)\big)\, .$$
  A regular form is called \emph{non-degenerate} if its characteristic distribution is trivial, that is, if its class is the total dimension of $M$. Otherwise, it is called \emph{degenerate}.
\end{definition}

\begin{proposition}\label{proposition:charact_dist_invariant}
  If $\alpha$ is a regular $k$-form on a supermanifold, then its characteristic distribution $\chi(\alpha)$ is involutive and $\alpha$ is $\chi(\alpha)$-invariant.
\end{proposition}
\begin{proof}
  For any pair of sections $X$ and $Y$ of $\chi(\alpha)= \ker \alpha \cap \ker \dd \alpha$,
  $$i_{[X, Y]} \alpha = [i_{X}, \liedv{Y}] \alpha = i_{X} i_{Y} \dd \alpha = 0\, , \quad i_{[X, Y]} \dd \alpha =0\, ,$$
  and
  $$\liedv{X} \alpha = \dd (i_{X} \alpha) + i_{X} \dd \alpha = 0\, .$$
\end{proof}

\mn Hence, $\chi(\alpha)$ can be viewed as a foliation. Indeed, on a neighbourhood of each $x_0\in \vert M \vert$,  the Homogeneous Frobenius theorem (Theorem~\ref{theorem:Frobenius}) provides a system of homogeneous coordinates $(x^a)$ such that the projection $\pi\colon M \to M/\chi(\alpha)$ onto the space of leaves reads
$$\pi(x^1, \ldots, x^n) = (x^{r+1}, \ldots x^n)\, ,$$
with $n$ the total dimension of $n$, and $r$ the total rank of $\chi(\alpha)$.
Furthermore, $\alpha$ induces a local non-degenerate $k$-form $\alpha_{\mathrm{red}}$ on  $M/\chi(\alpha)$ uniquely determined by
$$\pi^\ast (\alpha_{\mathrm{red}}) = \alpha\, .$$
This reduced form is $w$-homogeneous with respect to the reduced weight vector field, i.e.,
$$\liedv{\n_\red} \alpha_\red = \liedv{\pi_\ast \n} (\pi^\ast \alpha) = \pi^\ast \left( \liedv{\n} \alpha \right) = w\cdot \alpha_\red\, .$$
Therefore, in order to construct Darboux coordinates for $\alpha$, we can first construct them for the non-degenerate form $\alpha_\red$, and then pullback by $\pi$.

\subsection{Homogeneous presymplectic forms}

\begin{definition}
   A \emph{presymplectic form} $\omega$ of on a supermanifold $M$ is a closed $2$-form of constant rank, i.e., the vector bundle morphism
   \begin{align*}
    \flat_\omega \colon \T M & \to \cT M\\
    v & \mapsto i_{v} \omega
   \end{align*}
   has constant rank. If this morphism is an isomorphism, it is called a \emph{symplectic form}. The pair $(M, \omega)$ is called a \emph{(pre)symplectic manifold}.
\end{definition}

\no Note that, since it is closed, the class of a presymplectic form is just its rank.

\mn Let us recall a fundamental result from \cite{Grabowska:2025}:

\begin{theorem}[Darboux theorem for homogeneous symplectic forms]\label{theorem:Darboux_symplectic}
  Let $\omega$ be a homogeneous symplectic form of degree $\lambda = (\sigma, w) \in \ZZ_2\times \RR$ on a homogeneity supermanifold $(M, \n)$. Then, around any point $x_0\in \vert M \vert$, there exists a system of homogeneous coordinates $(q^1, \ldots, q^r, p_1, \ldots p_r,$ $y^1, \ldots y^s)$ such that
    $$\omega = \sum_{i=1}^r \dd p_i \wedge \dd q^i + \sum_{l=1}^s  \epsilon^l \dd y^l \wedge \dd y^l\, ,$$
    where $\epsilon^l = \pm 1$. The coordinates $y^l$ appear only if $\omega$ is even (i.e., $\sigma = 0$).
\end{theorem}

\no Combining this theorem and the Homogeneous Frobenius theorem (Theorem~\ref{theorem:Frobenius}), we can prove the following.

\begin{corollary}[Darboux theorem for homogeneous presymplectic forms]\label{theorem:Darboux_presymplectic}
  Let $\omega$ be a homogeneous presymplectic form of rank $2r+s$ and degree $\lambda = (\sigma, w) \in \ZZ_2\times \RR$ on a homogeneity supermanifold $(M, \n)$. Then, around any point $x_0\in \vert M \vert$ such that either $\n(m)=0$ or $\n(m)\neq 0$ and $\n(m)\notin \ker \omega_m$, there exists a system of homogeneous coordinates $(q^1, \ldots, q^r, p_1, \ldots, p_r, y^1, \ldots, y^s,$ $z^1, \ldots, z^{k})$ such that
  $$\omega = \sum_{i=1}^r \dd p_i \wedge \dd q^i + \sum_{l=1}^s  \epsilon^l \dd y^l \wedge \dd y^l\, ,$$
  where $k=\dim M -2r-s$ and $\epsilon^l = \pm 1$. The coordinates $y^l$ appear only if $\omega$ is even (i.e., $\sigma = 0$).
\end{corollary}

\begin{remark}
  Homogeneous Darboux coordinates for $\omega$ may not exist on a neighbourhood of a point $m\in |M|$ such that $\n(m)\neq 0$ and $\n(m) \in \ker\omega_m$.
\end{remark}

\begin{counterexample}\label{example:not_exist_Darboux_presymp}
  Let $M=\RR^3$, with canonical coordinates $(q, p, z)$. The vector field
   $$\n = \pa_z - \sin(q) \pa_q + p \cos(q) \pa_p$$
  is a weight vector field, because it is nowhere-vanishing. The two-form $\omega = \dd p \wedge \dd q$ is a presymplectic form of corank $1$, with $\ker \omega = \gen{\pa_z}$, and it is homogeneous of weight $0$.

  On the quotient space $N = \RR^3/\ker(\omega) \cong \{z = 0\}$, with natural coordinates $(q, p)$, we have the symplectic form $\tilde{\omega} = \dd p \wedge \dd q$. However, $\n$ does not induce a weight vector field on the quotient. Indeed,
  $$\tilde{\n} \coloneqq \pi_\ast \n = - \sin(q) \pa_q + p \cos(q) \pa_p\, ,$$
  where $\pi\colon M\to N$ is the natural projection, is not a weight vector field. The general solution $f_w$ of the differential equation
  $$\tilde{\n} f_w = w \cdot f_w\, ,\quad w\in \RR\, ,$$
  is given by
  $$f_w(q, p) = \left(\frac{\cos(q/p)}{\sin(q/2)}\right)^{w} F(p\sin q)\, ,$$
  for some function $F$. In particular, $f_w$ is not well-defined at $(0,0)$ for $w\neq 0$, while
  $$\dd f_0 (0, 0) = 0\, .$$
  Thus, there exist no homogeneous coordinates around the point $(q,p)=(0,0)$. Consequently, there do not exist homogeneous Darboux coordinates for $\omega$ (if they existed, they would induce homogeneous Darboux coordinates for   $\tilde{\omega}$).
\end{counterexample}

\subsection{Homogeneous one-forms of constant class}

Let $\alpha$ be a non-zero regular one-form on a supermanifold $M$. The annihilator of the characteristic distribution is given by
$$\big(\chi(\alpha)\big)^\circ = (\ker \alpha)^\circ + (\ker \dd \alpha)^\circ = \gen{\alpha} + \Im(\flat_{\dd \alpha})\, ,$$
where $\flat_{ \dd \alpha} \colon \T M \ni v \mapsto i_{v} \dd \alpha\in \cT M$. There are three possible cases:
\begin{enumerate}
  \item if $\dd \alpha = 0$, then $\class(\alpha) = 1$;
  \item if $\gen{\alpha} \cap \Im(\flat_{\dd \alpha}) = \{0_M\}$, where $0_M$ denotes the zero section of $\T M$, then $\class(\alpha) = \rank(\flat_{\dd \alpha}) + 1$;
   \item if $\gen{\alpha} \subseteq \Im(\flat_{\dd \alpha})$, then $\class(\alpha) = \rank(\flat_{\dd \alpha})$.
\end{enumerate}
In particular, $\alpha$ is non-degenerate if and only if $\big(\chi(\alpha)\big)^\circ = \cT M$, and the possible cases are now:
\begin{enumerate}
  \item $\cT M = \gen{\alpha} \oplus \Im(\flat_{\dd \alpha}) \text{ or } \dim M = 1 \Longrightarrow \dim M = \rank (\flat_{\dd \alpha}) +1$,
  \item $\cT M = \Im(\flat_{\dd \alpha})\Longrightarrow \dim M = \rank (\flat_{\dd \alpha})$.
\end{enumerate}
Here $\dim M$ denotes the total dimension of $M$.

\mn It is now easy to see (cf. \cite[Chapter VI]{Godbillon:1969}) that, on purely even manifolds, our class of a Pfaffian form coincides with the one defined by Darboux.
\begin{definition}
  A non-degenerate one-form $\alpha$ on a supermanifold $M$ is called a \emph{contact form} if $\dim M = \rank (\flat_{\dd \alpha}) +1$, and a \emph{symplectic potential} if $\dim M = \rank (\flat_{\dd \alpha})$.
\end{definition}
\no The definition of contact form is equivalent to the one in Section 4 from \cite{Grabowski:2013}. On the other hand, a symplectic potential is just a one-form whose differential is a symplectic form.
It is straightforward to verify the following.
\begin{proposition}\label{proposition:Reeb}
  Let $\alpha$ be a homogeneous contact form of degree $\lambda = (\sigma, w)\in \ZZ_2 \times \RR$ on a homogeneity supermanifold $(M, \n)$. Then, there is a unique vector field $\Reeb$ on $M$ such that
  $$i_{\Reeb} \alpha =1\, , \quad \Reeb \in \ker \dd \alpha\, .$$
  This vector field is nowhere-vanishing and homogeneous of degree $(\sigma, -w)$. It is called the \emph{Reeb vector field} of $\alpha$.
\end{proposition}

\begin{definition}
  A regular one-form $\alpha$ on a supermanifold $M$ is called a \emph{precontact form} (respectively, a \emph{presymplectic potential}) if the (local) reduced one-form $\alpha_\red$ on $M/\chi(\alpha)$ is a contact form (respectively, a symplectic potential).
\end{definition}
\no The reduced one-form $\alpha_\red$ may not be defined globally, because the foliation $\pi\colon M\to M/\chi(\alpha)$ may not be simple. Nevertheless,
the notions of precontact form and presymplectic potential are well-defined, as, by definition, the characteristic distribution of a regular one-form has constant rank.

\begin{theorem}[Darboux theorem for homogeneous one-forms]\label{theorem:Darboux_one_forms}
   Let $\alpha$ be a regular homogeneous one-form of degree $\lambda = (\sigma, w)\in \ZZ_2 \times \RR$ on a homogeneity supermanifold $(M, \n)$.

   Around each point $x_0\in \vert M \vert$ such that either $\n(m)=0$ or $\n(m)\neq 0$ and $\n(m)\notin \chi(\alpha)_m$, the one-form $\alpha$ has one of the following canonical expressions:
   \begin{enumerate}
    \item If $\alpha$ is a precontact form of class $2r+s+1$ and $\n(x_0) = 0$ or $w\neq 0$,
    $$\alpha=\dd z +  \sum_{i=1}^r  p_i \dd q^i + \sum_{l=1}^s  \epsilon^l y^l \dd y^l, \quad \epsilon^l= \pm 1\, ,$$
    in a certain system of homogeneous coordinates $(q^i, p_i, z, y^l, x^a)$.
    \item If $\alpha$ is a precontact form of class $2r+s+1$, $\n(x_0) \neq 0$ and $w= 0$,
    \begin{equation}\label{eq:canonical_contact_zero}
      \alpha=\frac{\dd z}{z}+  \sum_{i=1}^r  p_i \dd q^i + \sum_{l=1}^s  \epsilon^l y^l \dd y^l, \quad \epsilon^l= \pm 1\, ,
    \end{equation}
    in a certain system of homogeneous coordinates $(q^i, p_i, z, y^l, x^a)$.
    \item If $\alpha$ is a presymplectic potential of class $2r+s$,
    $$\alpha = \sum_{i=1}^r  p_i \dd q^i  + \sum_{l=1}^s  \epsilon^l y^l \dd y^l, \quad \epsilon^l= \pm 1\, ,$$
    in a certain system of homogeneous coordinates $(q^i, p_i, y^l, x^a)$
   \end{enumerate}
   The coordinates $(y^l)$ only appear if $\alpha$ is even, i.e., $\sigma = 0$.
\end{theorem}

\begin{remark}
  We do not consider one-forms with singularities. The homogeneous coordinate $z$ appearing in \eqref{eq:canonical_contact_zero} is nowhere-vanishing on its domain. Despite the resemblance on the expression, $\alpha$ should not be confused with a b-contact form (\textit{cf.} \cite[Theorem 5.4]{Miranda:2023}).
\end{remark}

\begin{proof}
  If $\alpha$ is degenerate, then we can obtain the canonical expression for the (local) reduced form $\alpha_\red$ on $M/\chi(\alpha)$ and pullback by $\pi\colon M \to M/\chi(\alpha)$ (see the argument below Proposition~\ref{proposition:charact_dist_invariant}).
  Consequently, the rest of the proof will be devoted to non-degenerate forms.

  \mn If $\alpha$ is an even contact form, then $\dd \alpha$ is an even homogeneous presymplectic form. By Corollary~\ref{theorem:Darboux_presymplectic}, there exists a homogeneous system of coordinates $(q^i, p_i, \overline{z}, \xi^l)$ around $x_0$ such that
  $$\dd \alpha = \sum_{i=1}^r \dd p_i \wedge \dd q^i + \sum_{l=1}^s  \epsilon^l \dd \xi^l \wedge \dd \xi^l\, ,$$
  where $\epsilon^l = \pm 1$. Therefore, the one-form
  $$\alpha' = \alpha - \sum_{i=1}^r  p_i \dd q^i - \sum_{l=1}^s  \epsilon^l \xi^l \dd \xi^l$$
  is closed and $\lambda$-homogeneous. There are two different cases (see Lemma~\ref{lemma:Poincare}):
  \begin{enumerate}
    \item If $\nabla(x_0)=0$ or $w\neq 0$, there exists a $\lambda$-homogeneous function $z$ such that $\dd z = \alpha'$. Therefore,
    \begin{equation}\label{eq:canonical_contact}
      \alpha = \dd z +  \sum_{i=1}^r  p_i \dd q^i + \sum_{l=1}^s  \epsilon^l \xi^l \dd \xi^l\, .
    \end{equation}
    \item If $\nabla(x_0)\neq 0$ and $w=0$, we can apply the standard Poincaré Lemma, obtaining a (not homogeneous) function $f$ such that $\dd f = \alpha'$ and $f(x_0) = 0$. Therefore,
    \begin{equation}\label{eq:precanonical_contact_zero}
    \alpha = \dd f +  \sum_{i=1}^r  p_i \dd q^i + \sum_{l=1}^s  \epsilon^l \xi^l \dd \xi^l\, .
    \end{equation}
    The fact that $\alpha$ and $\dd \alpha$ have weight $0$ implies that
    $$\w(q^i) + \w(p_i) = 0 = \w(\xi^l)\, ,\quad \forall\, i,l\, ,$$
    and then
    $$0 = \liedv{\nabla} \alpha = \liedv{\nabla} \dd f\, .$$
    Thus, $\n(f) = v$ for some constant $v\in \RR$. The function $z=e^f$ is homogeneous of weight $v$, and the contact form can be written as
    $$ \alpha = \frac{\dd z}{z} +  \sum_{i=1}^r  p_i \dd q^i + \sum_{l=1}^s  \epsilon^l \xi^l \dd \xi^l\, . $$
  \end{enumerate}
  In both cases, it just remains to verify that $z$ is a coordinate, which follows from the fact that $\cT M = \gen{\alpha} \oplus \Im (\flat_{\dd \alpha})$ and that $\dd \alpha$ is independent of $z$. If $\alpha$ is an odd contact form, then $\dd \alpha$ is an odd homogeneous presymplectic form, and the proof is analogous.

  \mn If $\alpha$ is an even symplectic potential, then we can construct the even contact form
  $$\eta = \dd t + \alpha\, ,$$
  on $M\times \RR$,  where $t$ is the canonical coordinate on $\RR$. We can endow $M\times \RR$ with the weight vector field
  $$\overline{\n} = \n + w t \pa_t\, ,$$
  making $\eta$ homogeneous of degree $\lambda=(\sigma, w)$.
  If $\n(x_0) = 0$ or $w\neq 0$, there exists a system of homogeneous coordinates $(q^i, p_i, z, \xi^l)$ around $(x_0, 0)$ where $\eta$ is written like the right-hand side of \eqref{eq:canonical_contact}. The Reeb vector field of $\eta$ is given by $\Reeb = \pa_{t} = \pa_{z}$. It satisfies $i_{\Reeb} \alpha = 0$ and $i_{\Reeb} \dd \alpha = 0$, which means that $\alpha$ is independent of the coordinate $z$. In addition, $z=t$, since $i_{\pa_z} \dd t = i_{\Reeb} \dd t=1$ and $z(x_0, 0) = 0$. Consequently, the coordinates $(q^i, p_i, \xi^l)$ project into coordinates on $M$, which are canonical coordinates for $\alpha$.

  \mn A completely analogous technique can be used for the case $\n(x_0) \neq 0$ and $w = 0$, by considering the (not homogeneous) system of coordinates $(q^i, p_i, f, \xi^l)$ where $\eta$ is written like the right-hand side of \eqref{eq:precanonical_contact_zero} and $\Reeb = \pa_f$. Once again, the coordinates $(q^i, p_i, \xi^l)$ are homogeneous and project into coordinates on $M$, which are canonical coordinates for $\alpha$.

  \mn Finally, if $\alpha$ is an odd symplectic potential, a similar technique can be employed by constructing the odd contact form $\eta = \dd \xi + \alpha$ on $M\times \RR^{0\vert 1}$, with $\xi$ the canonical coordinate on $\xi$.

\end{proof}

\begin{counterexample}
  Homogeneous Darboux coordinates for $\alpha$ may not exist around a point $m\in |M|$ such that $\n(m)\neq 0$ and $\n(m) \in \chi(\alpha)_m$.
  \mn Let $M=\RR^3$, with canonical coordinates $(q, p, z)$, equipped with the weight vector field $\n$ and the presymplectic potential $\omega$ from Example~\ref{example:not_exist_Darboux_presymp}. Consider the presymplectic potential $\alpha = p\, \dd q$. Note that $\alpha$ is homogeneous of weight $0$ and $\dd \alpha = \omega$. If there existed homogeneous Darboux coordinates for $\alpha$, in particular they would be homogeneous Darboux coordinates for $\omega$. Hence, $\alpha$ does not have homogeneous Darboux coordinates around $m=(0,0,0)$.
\end{counterexample}
\no From either the definitions or the local expressions, it is straightforward to verify the following.

\begin{proposition}
  A nowhere-vanishing one-form $\alpha$ on a supermanifold $M$ is regular if and only if its differential $\dd \alpha$ is a presymplectic form. More precisely,
  \begin{itemize}
    \item $\alpha$ is a precontact form of class $2r+s+1$ if and only if $\dd \alpha$ is a presymplectic form of rank $2r+s$ and $\gen{\alpha} \cap \Im(\flat_{\dd \alpha}) = \{0_M\}$,
    \item $\alpha$ is a presymplectic potential of class $2r+s$ if and only if $\dd \alpha$ is a presymplectic form of rank $2r+s$ and $\gen{\alpha} \subseteq \Im(\flat_{\dd \alpha})$.
  \end{itemize}
\end{proposition}

\begin{remark}
  The zero section of the tangent bundle $\n=0_M$ makes any supermanifold a homogeneity supermanifold. Consequently, Theorems \ref{theorem:Darboux_presymplectic} and \ref{theorem:Darboux_one_forms} provide the canonical local expressions for any presymplectic form, precontact form or presymplectic potential on a supermanifold, as long as they have (locally) constant class and a defined parity.
\end{remark}

\section{Examples of homogeneous Pfaffian forms}\label{sec:examples}

\subsection{Standard Darboux theorems}\label{subsec:standard_Darboux}

Any (super)manifold can be regarded as a homogeneity (super)manifold by endowing it with the trivial weight vector field $\n \equiv 0$. Hence, the standard Darboux theorems for presymplectic forms and regular one-forms are particular cases of Theorems~\ref{theorem:Darboux_presymplectic} and \ref{theorem:Darboux_one_forms}, respectively. More precisely, we have the following:

\begin{proposition}[Darboux theorems on supermanifolds]\label{proposition:Darboux_standard}
  Let $M$ be a supermanifold, and let $x_0\in \vert M \vert$ be an arbitrary point. Consider a regular one-form $\alpha$ and a presymplectic potential $\omega$ with a well-defined parity $\sigma\in \ZZ_2$ (i.e., they are either even or odd).
   \begin{enumerate}
    \item If $\rank \omega = 2r +s$, there exists a system of homogeneous coordinates $(q^i, p_i, y^l, x^a)$ such that
    $$\omega = \sum_{i=1}^r \dd p_i \wedge \dd q^i + \sum_{l=1}^s  \epsilon^l \dd y^l \wedge \dd y^l\, ,$$
    where $k=\dim M -2r-s$ and $\epsilon^l = \pm 1$.
    \item If $\alpha$ is a precontact form of class $2r+s+1$, there exists a system of local coordinates $(q^i, p_i, z, y^l, x^a)$ around $x_0$ such that
    $$\alpha=\dd z +  \sum_{i=1}^r  p_i \dd q^i + \sum_{l=1}^s  \epsilon^l y^l \dd y^l, \quad \epsilon^l= \pm 1\, .$$
    \item If $\alpha$ is a presymplectic potential of class $2r+s$,  there exists a system of local coordinates $(q^i, p_i, y^l, x^a)$ around $x_0$ such that
    $$\alpha = \sum_{i=1}^r  p_i \dd q^i  + \sum_{l=1}^s  \epsilon^l y^l \dd y^l, \quad \epsilon^l= \pm 1\, .$$

   \end{enumerate}
   The coordinates $(y^l)$ only appear if $\alpha$ is even, i.e., $\sigma = 0$.
\end{proposition}
\no As a particular case, for a contact form we recover \cite[Theorem 4.2]{Grabowski:2013}.

\subsection{Non-trivial homogeneous forms of weight zero}

Consider $\RR^3$ with the Cartesian coordinates $(x, y, z)$, with degrees $1, -1$ and $0$, respectively. In other words, $\RR^3$ is endowed with the weight vector field $\n = x \pa_x - y \pa_y$. Let $\varphi_i\colon\RR^2\to \RR$ be arbitrary smooth functions. Then, any functions $f_i\colon \RR^3 \to \RR$ of the form
$$f_i(x, y, z) = \varphi_i(xy, z)$$
are homogeneous of weight $0$. Assume that
$$\restr{\frac{\partial}{\partial v} \varphi_3(u, v)}{(0,0)} \neq 0\, .$$
Then, any functions of the form
$$q = x (1+f_1)\, , \quad p = y (1 + f_2)\, , \quad \zeta = f_3$$
form a new system of homogeneous coordinates around $(0,0,0)$, which also have weights $\{1, -1, 0\}$, i.e.,~$\n=q\pa_q - p \pa_p$.
Then,
$$\eta = \dd \zeta + p \dd q$$
is a (local) homogeneous contact form of weight $0$, and $(q, p, \zeta)$ are homogeneous Darboux coordinates around the origin. In canonical coordinates $(x, y, z)$, the expression of $\eta$ can be quite involved. For instance, for
$$\varphi_1(u,v) = \cos u\, , \quad \varphi_2(u,v) = \sin v\, , \quad \varphi_3(u,v) = e^v \cos u\, ,$$
we have that
\begin{align*}
  \eta & = y \Big(1+\sin z + \cos (x y) (1+\sin z) -\sin (x y) \big(e^z+x y (1+\sin z)\big) \Big)\, \dd x \\
  & \quad -x \sin (x y) \Big(x y (\sin z+1)+e^z\Big)\, \dd y + e^z \cos (x y)\,  \dd z \, .
\end{align*}
Observing this expression, the existence of coordinates $(q, p, \zeta)$ in which both $\eta$ and $\n$ have a canonical form is a priori not obvious. Nevertheless, the existence of such coordinates is guaranteed by Theorem~\ref{theorem:Darboux_one_forms}.

\mn
Similarly, the one-form
\begin{equation}\label{eq:example_symp_potential_weight_zero}
  \theta = y \big(\cosh (x y)+1\big) \big(\sinh (x y)+x y \cosh (x y)+1\big)\, \dd x
  + x^2 y \cosh (x y) \big(\cosh (x y)+1\big)\, \dd y
\end{equation}
on $\RR^2$ is homogeneous of weight $0$ with respect to  $\n = x \pa_x - y \pa_y$. Moreover, one can verify that it is a symplectic potential. Theorem~\ref{theorem:Darboux_one_forms} thus ensures the existence of an atlas formed by systems of homogeneous Darboux coordinates $(q, p)$ such that $\theta = p\, \dd q$. In this case, an explicit change of homogeneous coordinates is given by
\begin{equation}\label{eq:example_symp_potential_weight_zero_coords}
  q = x \big(1+ \sinh (x y)\big)\, , \quad p = y \big(1+ \cosh (x y)\big)\, .
\end{equation}

\subsection{Exact symplectic forms}

A symplectic form $\omega$ on a manifold $M$ is exact if there exists a one-form $\alpha$ such that $\omega = \dd \alpha$. Since $\omega$ defines a $\Cinfty(M)$-module isomorphism between $\X(M)$ and $\Omega^1(M)$, one can define a unique vector field $\n_\alpha\in \X(M)$ by $i_{\n_\alpha} \omega = \alpha$, which is called the Liouville vector field. Conversely, given a vector field $X$ such that $\liedv{X} \omega = \omega$, one can define a symplectic potential $\tilde{\alpha} \coloneqq i_{X} \omega$.

By the standard Darboux theorem (see Proposition~\ref{proposition:Darboux_standard}), there exist local coordinates $(q^i, p_i)$ in which $\alpha$ reads
$$\alpha = \sum_{i=1}^r  p_i \dd q^i$$
Thus,
$$\omega = \dd \alpha = \sum_{i=1}^r  \dd p_i\wedge \dd q^i\, , \quad \n_\alpha = \sum_i p_i \pa_{p_i}\, .$$
In particular, $\n_\alpha$ is of the form~\eqref{eq:weight_vector_field}. Thus,
$\n_\alpha$ is a weight vector field, and $\alpha$ is homogeneous of weight $1$ with respect to it.

In the case of the cotangent bundle $\cT N$ of a manifold $N$ equipped with the canonical one-form $\alpha_N$, the corresponding Liouville vector field $\n_{\alpha_N}$ is the infinitesimal generator of homotheties on the fibers, i.e., the Euler vector field $\n_{\mathcal{E}}$. Any exact symplectic manifold $(M, \alpha)$ is locally isomorphic to $(\cT N, \alpha_N)$. More precisely, around each point of $M$, there exists a local diffeomorphism $\Phi\colon M \to \cT N$
such that
$$\Phi^\ast \alpha_N = \alpha\, , \quad \Phi_\ast \n_{\alpha} = \n_{\mathcal{E}}\, .$$
Indeed, let $(x^i)$ be a system of local coordinates on $N$, and let $(x^i, \xi_i)$ be the induced system of bundle coordinates on $\cT N$. Then, we can write
$$\alpha_N = \sum_i \xi_i \dd x^i\, , \quad \n_{\mathcal{E}} = \sum_i \xi_i \pa_{\xi_i}\, ,$$
and consider the local diffeomorphism
$$\Phi(q^i, p_i) = (x^i, p_i)\, .$$

\mn
Furthermore, a homogeneity manifold $(M, \n)$ equipped with a homogeneous symplectic potential $\alpha$ is a natural example of a bi-homogeneity manifold, i.e.,~a manifold equipped with two commuting weight vector fields. Indeed, it is straightforward to check (e.g.,~using a system of homogeneous Darboux coordinates) that $\n$ commutes with the Liouville vector field $\n_{\alpha}$ defined by $\alpha$. For instance, the homogeneous symplectic form $\omega = \dd \theta$, with $\theta$ defined in~\eqref{eq:example_symp_potential_weight_zero}, is homogeneous of weight $0$ with respect to $\n$, and of weight $1$ with respect to $\n_\theta$, where
$\n=q\pa_q - p\pa_p$ and $\n_\theta = p \pa_p$ in coordinates~\eqref{eq:example_symp_potential_weight_zero_coords}.

\subsection{Homogeneous Pfaffian equations}

A Pfaffian equation is a differential equation that can be expressed as the distribution given by the kernel of a regular one-form $\alpha$ on a manifold $M$, namely, $D=\ker \alpha$. The class of $\alpha$ characterizes the integrability of the associated Pfaffian equation.

If $(M, \n)$ is a homogeneity manifold, and $\alpha$ is homogeneous, then $D$ is a homogeneous distribution. Homogeneous Darboux coordinates for $\alpha$ provide the solutions for $D$, i.e., its maximal integral submanifolds:
\begin{itemize}
    \item If $\alpha$ is a presymplectic potential of class $2r+s$,
    $$\alpha = \sum_{i=1}^r  p_i \dd q^i\, ,$$
   and the submanifolds given by $\{q^i = c_i,\, z = c_0\}$ for constants $c_0,c_i\in \RR$ are solutions of $D$.
    \item If $\alpha$ is a precontact form of class $2r+s+1$ and $\n(x_0) = 0$ or $w\neq 0$,
    $$\alpha=\dd z +  \sum_{i=1}^r  p_i \dd q^i\, ,$$
    and $\{q^i = c_i,\, z = c_0\}$ for constants $c_0,c_i\in \RR$ are solutions of $D$.
    \item If $\alpha$ is a precontact form of class $2r+s+1$, $\n(x_0) \neq 0$ and $w= 0$,
    $$ \alpha=\frac{\dd z}{z}+  \sum_{i=1}^r  p_i \dd q^i\, ,$$
   and $\{q^i = c_i,\, z = c_0\}$ for constants $c_0,c_i\in \RR$ are solutions of $D$.
\end{itemize}
Note that the flow of $\n$ maps each particular solution onto another solution.

For instance, the First Law of Thermodynamics states that equilibrium configurations of a thermodynamic system are the solutions to the Pfaffian equation $D=\ker \alpha$, where $\alpha$ is the Gibbs contact form
$$\alpha = \dd U - \sum_i p_i \dd q^i\, ,$$
where $U$ is the internal energy, and $(q^i, p_i)$ are thermodynamic conjugate variables (such as volume and pressure, or number of molecules and chemical potential). Thermodynamically meaningful variables are either intensive or extensive.
We can declare the intensive (resp.~extensive) variables $p_i$ (resp.~$q^i$ and $z$) to be homogeneous of weight $0$ (resp.~$1$), i.e., define the weight vector field $\n = U \pa_U + \sum_i p_i \pa_{p_i}$. In this way, any other system of homogeneous coordinates consists of variables that are either intensive or extensive.

\section*{Acknowledgements}

The authors thank Paweł Urbański for his valuable comments on this manuscript. They are also grateful to the referee for their suggestions.

\section*{Declarations}

\subsection*{Funding}
The research of JG was funded by the National Science Centre (Poland) within the project WEAVE-UNISONO, No. 2023/05/Y/ST1/00043.

\subsection*{Competing interests}

The authors have no competing interests to declare.

\bibliographystyle{plain}
\bibliography{biblio-final}

\vskip.5cm
\noindent Janusz Grabowski\\\emph{Institute of Mathematics, Polish Academy of Sciences}\\{\small ul. \'Sniadeckich 8, 00-656 Warszawa,
Poland}\\{\tt jagrab@impan.pl}\\  https://orcid.org/0000-0001-8715-2370
\\

\noindent Asier L\'opez-Gord\'on\\\emph{Institute of Mathematics, Polish Academy of Sciences}\\{\small ul. \'Sniadeckich 8, 00-656 Warszawa,
Poland}\\{\tt alopez-gordon@impan.pl}\\  https://orcid.org/0000-0002-9620-9647
\\

\end{document}